\documentclass{amsart}
\usepackage{amssymb}

\usepackage{amsfonts}
\usepackage[all,2cell]{xy} \UseAllTwocells \SilentMatrices
\usepackage{graphicx}
\usepackage{amsthm}
\usepackage{amstext}
\usepackage{amsmath}
\usepackage{amscd}
\usepackage[mathscr]{eucal}
\usepackage{url}

\newtheorem*{theorem*}{Theorem}
\newtheorem{theorem}{Theorem}[section]
\newtheorem{proposition}[theorem]{Proposition}
\newtheorem{corollary}[theorem]{Corollary}
\newtheorem{lemma}[theorem]{Lemma}

\theoremstyle{definition}
\newtheorem{definition}[theorem]{Definition}
\newtheorem{example}[theorem]{Example}

\theoremstyle{remark}
\newtheorem{remark}[theorem]{Remark}

\numberwithin{equation}{section}

\newcommand{\gc}{\Gamma_c(x, f)}

\def\N{{\mathbb N}}
\def\Z{{\mathbb Z}}
\def\Q{{\mathbb Q}}
\def\R{{\mathbb R}}

\def\U{{\mathcal U}}

\begin{document}

\begin{large}

\title{On Expansive Maps of Topological Spaces}
\author[A. Barzanouni, E. Shah]{Ali Barzanouni and Ekta Shah}
\address{Department of Mathematics, School of Mathematical Sciences, Hakim Sabzevari University, Sabzevar, Iran}
\address{Department of Mathematics, Faculty of Science, The Maharaja Sayajirao University of Baroda, Vadodara, India}
\email{barzanouniali@gmail.com}
\email{ekta19001@gmail.com; shah.ekta-math@msubaroda.ac.in}

\maketitle

\begin{abstract}
We show that if there exists a topologically expansive homeomorphism on a uniform space, then the space is always a regular space. Through examples we show that in general composition of topologically expansive homeomorphisms need not be topological expansive and also that conjugate of topologically expansive homeomorphism need not be topological expansive. Further, we obtain a characterization of orbit expansivity in terms of topological expansivity and conclude that if there exists a topologically expansive homeomorphism on a compact uniform space then the space must be metrizable. We also study positively expansive maps on topological space and obtain condition for maps  to be positively topological expansive in terms of finite open cover.  Further, we show that if there exists a continuous, one-to-one, positively topological expansive map on  a compact uniform space, then the space is finite. We also give an example of a positively topological expansive map on a non--Hausdorff space.
\end{abstract}
\section{Introduction}

\bigskip

\noindent  A homeomorphism  $h : X \longrightarrow X$ defined on metric space $X$ is said to be an \textit{expansive homeomorphism} provided there exists a real number $c>0$ such that whenever $x, y \in X$ with $x\not=y$ then there exists an integer $n$ (depending on $x, y$) satisfying  $d(h^n(x), h^n(y))>c$. Constant $c$ is called an \textit{expansive constant} for $h$. In 1950, Utz, \cite{U}, introduced it with the name unstable homeomorphisms. The examples discussed in this paper on compact spaces were sub dynamics of shift maps, thus one can say that the theory of expansive homeomorphisms started based on symbolic dynamics but it quickly developed by itself.

\bigskip

\noindent Much attention has been paid to the existence / non--existence of expansive homeomorphisms on given spaces. Each compact metric space that admits an expansive homeomorphism is finite-dimensional~\cite{rm}. The spaces admitting expansive homeomorphisms include the Cantor set, the real line/half-line, all open $n-$cells, $n\geq 2$ \cite{Kou}. On the other hand, spaces not admitting expansive homeomorphisms includes any Peano continuum in the plane \cite{HK-2}, the 2-sphere the projective plane and the Klein bottle \cite{Hi0}.

\bigskip
\noindent Another important aspects of expansive dynamical system is the study of its various generalizations and variations in different setting. The very first of such variation was given by Schwartzman, \cite{SS}, in 1952 in terms of positively expansive maps, wherein the points gets separated by non--negative iterates of the continuous map. In 1970, Reddy, \cite{wr}, studied point--wise expansive maps whereas $h-$expansivity was studied by R. Bowen, \cite{rb}. Kato defined and studied the notion of continuum--wise expansive homeomorphism \cite{HK}. Second author of the paper studied notion of positive expansivity of maps on metric $G-$spaces \cite{es} whereas the first author studied finite expansive homeomorphisms \cite{ali}. Tarun Das \emph{et al.} \cite{T.D} used the notion of expansive homeomorphism on topological space to prove the Spectral Decomposition Theorem on non--compact spaces. Achigar \emph{et al.} studied the notion of orbit expansivity on non--Hausdorff space \cite{art}. We studied expansivity for group actions in \cite{ase}. In this paper we study expansive homeomorphisms on uniform spaces and positively expansive maps on topological spaces.


\medskip
\noindent In Section 2 we discuss preliminaries regarding uniform spaces and expansive maps on metric space required for the content of the paper.  To the best of our knowledge the notion of expansivity for uniform spaces was first studied in \cite{pram} in the form of topological positive expansive maps. In Section 3 of this paper we define and study expansive homeomorphism on uniform spaces. Through examples it is justified that topological expansivity is weaker than metric expansivity. We further show that if a uniform space admits a topologically expansive homeomorphism then space is always a Hausdorff space. The notion of orbit expansivity was first introduced in \cite{art}. A characterization of orbit expansivity on compact uniform spaces is obtained in terms of topological expansivity. As a consequence of this we conclude that if there is a topologically expansive homeomorphism on a compact uniform space then the space is always metrizable. In the Section 4 of the paper we study positive expansivity on topological spaces. Through examples it is justified that positively topological expansivity is weaker than positively metric expansivity. We obtain a characterization of positively topological expansive maps in terms of finite open cover. Further, we show that if  there exists a continuous, one-to-one, positively topological expansive map on  a compact uniform space, then the space is finite. We also obtain a condition under which $T_1-$space admits a positively topological expansive map.

\bigskip
\section{Preliminaries}
\medskip
\noindent In this Section we discuss basics required for the content of the paper.

\medskip
\noindent
\subsection{Uniform Spaces}

\medskip
\noindent Uniform spaces were introduced by A. Weil \cite{Wei} as a generalization of metric spaces and topological groups. Recall, in a uniform space $X$, the closeness of a pair of points is not measured by a real number, like in a metric space, but by the fact that this pair of points belong or does not belong to certain subsets of the cartesian product, $X\times X$. These subsets are called the \emph{entourages} of the uniform structure.

\bigskip
\noindent Let $X$ be a non-empty set. A relation on $X$ is a subset of $X\times X$.  If $U$ is a relation, then the \textit{inverse} of $U$ is denoted by $U^{-1}$ and is a relation given by
$$ U^{-1}=\{(y, x): (x, y)\in U\}.$$ A relation $U$ is said to be \emph{symmetric} if $U=U^{-1}$. Note that $U\cap U^{-1}$ is always a symmetric set. If $U$ and $V$ are relations, then the \emph{composite of} $U$ and $V$ is denoted by $U\circ V$ and is given by
$$U\circ V = \left\{(x, z) \in X\times X : \exists \; y \in X \mbox{ such that } (x, y)\in V \; \& \; (y, z)\in U\right\}.$$

\medskip
\noindent The set, denoted by $\bigtriangleup_X$,  given by  $\bigtriangleup_X = \{(x, x) : x \in X\}$ is called the \emph{identity relation} or \emph{the diagonal of} $X$. For every subset $A$ of $X$ the set $U[A]$ is a subset of $X$ and is given by  $U[A]=\{y \in X : (x, y)\in U,  \mbox{ for  some} \; x\in A\}$. In case if $A=\{x\}$ then we denote it by $U[x]$ instead of $U[\{x\}]$. We now recall the definition of uniform space.

\medskip
\begin{definition}
A \emph{uniform structure (or uniformity)} on a set $X$ is a non--empty collection $\U$ of subsets of $X\times X$ satisfying the following properties:
\begin{enumerate}
  \item If $U\in\mathcal{U}$, then $\bigtriangleup_X \subset U$.
  \item If $U\in\mathcal{U}$, then $U^{-1}\in\mathcal{U}$.
  \item If $U\in \mathcal{U}$, then $VoV\subseteq U$, for some $V\in\mathcal{U}$.
  \item If $U$ and $V$ are elements of $\mathcal{U}$, then $U\cap V\in\mathcal{U}$.
  \item If $U\in\mathcal{U}$ and $U\subseteq V\subseteq X\times X$, then $V\in \mathcal{U}$.
\end{enumerate}
The pair $(X, \mathcal{U})$ (or simply $X$) is called as a \emph{uniform space}.
\end{definition}

\medskip
\noindent Obviously every metric on a set $X$ induces a uniform structure on $X$ and  every uniform structure on a set $X$ defines a topology on $X$. Further, if the uniform structure comes from a metric, the associated topology coincides with the topology obtained by the metric. Also, there may be several different uniformities  on a set $X$. For instance,  the largest uniformity on $X$ is the collection of all subsets of $X\times X$ which contains $\bigtriangleup_X$ whereas the smallest uniformity on $X$ contains only $X\times X$. For more details on uniform spaces one can refer to \cite{kelly}.

\medskip
\begin{example}\label{compa}
  Consider $\R$ with usual metric $d$. For every $\epsilon>0$, let
$$U^d_\epsilon:=\left\{(x, y)\in \R^2 :  d(x, y)<\epsilon\right\}$$
Then the collection
$$  \U_d=\left\{E\subseteq \R^2 : U_\epsilon^d\subseteq E, \text{ \ for some  } \epsilon>0\right\}$$
is a uniformity on $\R$.  Further, let $\rho$ be an another metric on $\R$ given by $\rho(x,y) = | e^x -e^y |, \;\; x, y \in \R$. If for $\epsilon >0$,
$$U_\epsilon^\rho:=\left\{(x, y)\in \R^2 : \rho(x, y)<\epsilon\right\}$$
then the collection
$$  \U_{\rho}=\left\{E\subseteq \R^2 : U_\epsilon^\rho \subseteq E \text{ \  for some  } \epsilon>0\right\}$$
is also a uniformity on $\R$. Note that these two uniformities  are distinct as the set $\{ (x,y) : | x - y | < 1 \}$ is in $\U_d$  but it is not in $\U_{\rho}.$
\end{example}

\medskip
\noindent Let $X$ be a uniform space with uniformity $\U$. Then, the natural topology, $\tau_{\U}$, on $X$  is the family of all subsets $T$ of $X$ such that for every $x$ in $T$, there is $U\in\mathcal{U}$ for which $U[x]\subseteq T$.  Therefore, for each $U \in \U$,  $U[x]$ is a neighborhood of $x$. Further, the interior of a subset $A$ of $X$ consists of all those points $y$ of $X$ such that $U[y]\subseteq A$, for some $U\in\mathcal{U}$. For the proof of this, one can refer to  \cite[Theorem 4, P. 178]{kelly}. With the product topology on $X\times X$, it follows that every member of $\mathcal{U}$ is a neighborhood of $\Delta_X$ in $X\times X$. However, converse need not be true in general. For instance, in Example \ref{compa} every element of $\U_d$ is a neighborhood of $\Delta_\mathbb{R}$ in $\R^2$ but $\left\{(x, y) : |x-y|<\frac{1}{1+|y|}\right\}$ is a neighborhood of $\Delta_\mathbb{R}$ but not a member of $\U_d$.
Also, it is known that if $X$ is a compact uniform space, then $\U$ consists of all the neighborhoods of the diagonal $\Delta_X$ \cite{kelly}. Therefore for compact Hausdorff spaces the topology generated by different uniformities is unique and hence the only uniformity on $X$ in this case is the natural uniformity.  Proof of the following Lemma can be found in \cite{kelly}.

\medskip
\begin{lemma}\label{has}
  Let $X$ be a uniform space with uniformity $\U$. Then the following are equivalent:
  \begin{enumerate}
    \item $X$ is a $T_1-$space.
    \item $X$ is a Hausdorff space.
    \item $\bigcap\{U : U\in \U\}= \Delta_X$.
    \item $X$ is a regular space.
  \end{enumerate}

\end{lemma}


\medskip
\subsection{Various kind of Expansivity on Metric / Topological Spaces}

\noindent Let $X$ be a metric space with metric $d$ and let $f : X \longrightarrow X$ be a homeomorphism. For $x\in X$ and a positive real number $c$, set
$$\Gamma_c(x, f)=\left\{y : d(f^{n}(x), f^{n}(y))\leq c, \forall n\in \Z\right\}.$$
$\gc$ is known as \emph{the dynamical ball of $x$ of size $c$}. Note that for each $c$, $\gc$ is always non--empty. We recall the definition expansive homeomorphism defined by Utz \cite{U}.

\medskip
\begin{definition}
Let $X$ be a metric space with metric $d$ and let $f : X \longrightarrow X$ be a homeomorphism. Then $f$ is said to be a \emph{metric expansive homeomorphism}, if there exists $c>0$ such that $\gc=\{x\}$, for all $x\in X$. Constant $c$ is known as an \textit{expansive constant} for $f$.

\end{definition}

\medskip
\noindent The notion of metric expansive homeomorphism is independent of the choice of metric if the space is compact but not the expansive constant. If the space is non--compact then the notion of metric expansivity depends on the choice of metric even if the topology induced by different metrics are equivalent. For instance, see Example \ref{ex1}. Different variants and generalizations of expansivity are studied. We study few of them in this section.

\medskip




\medskip
\noindent Let $(X, \tau)$ be a topological space. For a subset $A\subseteq X$ and a cover $\mathcal{U}$ of $X$ we write $A\prec \mathcal{U}$ if there exists $C\in \mathcal{U}$ such that $A\subseteq C$. If $\mathcal{V}$ is a family of subsets of $X$, then $\mathcal{V}\prec \mathcal{U}$ means that $A\prec \mathcal{U}$ for all $A\in\mathcal{V}$. If, in addition $\mathcal{V}$ is a cover of $X$, then $\mathcal{V}$ is said to be \emph{refinement of} $\mathcal{U}$. Join of two covers $\mathcal{U}$ and $\mathcal{V}$ is a cover given by $\mathcal{U}\wedge \mathcal{V}=\{U\cap V| U\in\mathcal{U}, V\in\mathcal{V}\}$. Every open cover $\mathcal{U}$ of cardinality $k$ can be refined by an open cover $\mathcal{V}=\bigwedge_{i=1}^k \mathcal{U}$ such that $\mathcal{V}\prec \mathcal{U}$ and $\mathcal{V}\bigwedge\mathcal{V}=\mathcal{V}$.

\medskip
\noindent Let $X$ be an infinite $T_1-space$ and $\mathcal{U}$ be a finite open cover of $X$. Then there is open cover $\mathcal{V}$ such that $\mathcal{U} \nprec \mathcal{V}$. Therefore, if $f : X \longrightarrow X$ is a homeomorphism defined on $X$, then for every finite open cover $\mathcal{U}$ of $X$, there is an open cover $\mathcal{V}$ such that  $f^{-n}(\mathcal{U})\nprec \mathcal{V}$, for all $n\in\mathbb{N}$. Because if it is not true, then there is a finite open cover $\mathcal{U}$ of cardinality $k$ such that for  every   open cover $\mathcal{V}$ there is $n\in\mathbb{N}$ with $f^{-n}(\mathcal{U})\prec \mathcal{V}$. Take $k+1$ distinct point $\{x_n\}_{n=0}^k$ and $\mathcal{V}=\{ X-\{x_0, x_1, \ldots, \widehat{x_i}, \ldots, x_k\} : i=0, \ldots, k\}$ where "hat" means that the element must be omitted. Hence there is $n\in\mathbb{N}$ such that $f^{-n}(\mathcal{U})\prec \mathcal{V}$. That is a contradiction since $\mathcal{V}$ has no proper subcover.

\medskip
\begin{proposition}\label{finite}
Let $X$ be a $T_1$-space and $f:X\to X$ be a homeomorphism. The space $X$ is finite if and only if there is finite open cover $\mathcal{U}$ such that for every open cover $\mathcal{V}$ there is $n\in\mathbb{N}$ such that $f^{-n}(\mathcal{U})\prec \mathcal{V}$.
\end{proposition}

\medskip
\noindent Being $T_1-$space  is essential in Proposition \ref{finite}.  For instance, consider $X=[0,1)$ with the topology $\tau_X=\{(a, 1): a\geq 0\}\cup \{X, \emptyset\}$. It is clear that $(X, \tau_X)$ is  compact, not $T_1$. Take open cover $\mathcal{U}= \{X\}$. It is clear that if $\mathcal{V}$ is an open cover, then $X\in \mathcal{V}$, hence  $\mathcal{U}\prec \mathcal{V}$.

\medskip
\noindent The notion for orbit expansivity for homeomorphisms was first defined in \cite{art}. We recall the definition.

\medskip
 \begin{definition}
 Let  $f: X \longrightarrow X$ be a homeomorphism defined on a topological space $X$. Then $f$ is said to be \emph{an orbit  expansive homeomorphism} if there is a finite open cover $\mathcal{U}$ of $X$ such that if  for each $n \in \Z$, the set $\{f^n(x), f^n(y)\}\prec \mathcal{U}$, then $x=y$. The cover $\mathcal{U}$ of $X$  is called \emph{an orbit  expansive covering of} $f$.
 \end{definition}

\medskip
\noindent It can be observed that if $f$ is an orbit expansive homeomorphism on a compact metric space and $\mathcal{U}$ is an orbit expansive covering of $f$, then $\U$ is a generator for $f$ and therefore $f$ is an expansive homeomorphism. Conversely, every expansive homeomorphism on a compact metric space has a generator $\U$, which is also an orbit expansive covering of $f$. Hence on compact metric space expansivity is equivalent to orbit expansivity. Another generalization of expansivity was defined and studied in \cite{T.D}. We recall the definition.

\medskip
\begin{definition}
Let $X$ be a topological space. Then a homeomorphism $f : X\longrightarrow X$ is said to be an \emph{expansive homeomorphism} if there exists a  neighborhood $N$ of $\Delta_X$ such that for any two distinct $x, y\in X$, there is $n\in\mathbb{Z}$  satisfying $(f^n(x), f^n(y))\notin N$. Closed Neighbourhood $N$ is called \emph{an expansive neighborhood} for $f$.
\end{definition}

\medskip
\noindent Note that the term used in \cite{T.D} is topologically expansive but we used the term expansive in above definition to differentiate it from our definition of expansivity on uniform spaces.Obviously, metric expansivity implies expansivity. Through examples it was justified in \cite{T.D} that in general expansivity need not imply metric expansivity. Also, similar to proof of \cite[Theorem 4]{D.R}, one can show that on a locally compact metric space $X$, if $f:X \longrightarrow X$ is expansive with expansive neighborhood $N$, then  for every $\epsilon>0$ we can construct a metric $d$ compatible with the topology of $X$ such that $f$ is a metric expansive with expansive constant $\epsilon>0$.

\medskip
\noindent One of the early variation of metric expansivity was studied by Schwartzman in \cite{SS}, namely positively metric expansive. We recall the definition.

\medskip
\begin{definition}
Let $X$ be a metric space with metric $d$ and let $f : X \longrightarrow X$ be a continuous map. Then $f$ is said to be a \emph{positively metric expansive map}, if there exists $c>0$ such that for $x, y\in X$ with $x \neq y$ there is a non--negative integer $n$ satisfying $d(f^n(x), f^n(y)) >c$. Constant $c$ is known as an \textit{expansive constant for $f$}.
\end{definition}

\medskip
\noindent An example positively metric expansive map is  a doubling map on the unit circle. But it is known that there exists no metric expansive homeomorphism on the unit circle \cite{an}. Also, the shift map of the full two--sided shift is a metric expansive homeomorphism which is not positively metric expansive map. Hence in general, metric expansivity and positive metric expansivity are not related.

\bigskip
\noindent Note that it may be happen that for a positively metric expansive map $f:X \longrightarrow X$  it may happen to two distinct points are separated only by the zeroth iterated of the map $f$ but not separated by any other iterate. For instance, consider $X=\{x\in\Bbb R:|x|\ge1\}$ with the usual metric and defined $f : X \longrightarrow X$ by  $f(x)=2|x|$ for $x\in X$. Then $f$ is positively metric expansive with expansive constant $\delta$ such that $0< \delta < 1$. Now, let $x,y\in X$ be such that $x\ne y$ but $|x|=|y|$.  Then $d(x,y)>1$. However, $d(f^n(x),f^n(y))=0$ for all $n\ge 1$. Observe that $f$ is not a one--one map.

\bigskip
\noindent Suppose $f$ is a positively metric expansive one--one map defined on a metric space $X$ with expansive constant $c >0$. Let $x\neq y$. Then there exists $n>0$ such that $d(f^n(x),f^n(y))>c$.  Setting $x'=f^n(x)$ and $y'=f^n(y)$, there  exists some $n'>0$ such that $d(f^{n'}(x'),f^{n'}(y'))>c$.  But $f^{n'}(x')=f^{n+n'}(x)$ and $f^{n'}(y')=f^{n+n'}(y)$.  Thus if $n^{th}$ iterated of $f$ separates $x$ and $y$, then we can even find a larger integer $n+n'$ which also separates $x$ and $y$. Hence there are infinitely many positive integers $n$ which separates $x$ and $y$. But this is not true with if $f$ is a metric expansive homeomorphism on compact metric space $X$. For example, take $X=\{0\}\cup\{1/n:n\in\N\}$ with usual metric, and let $g:\mathbb{Z}\longrightarrow \{1/n:n\in\N\}$ be a bijection such that $g(0)=1$.  Define $f:X\longrightarrow X$ by $f(0)=0$ and $f(g(n))=g(n-1)$ for each $n\in\mathbb{Z}$.  Then $f$ is metric expansive with expansive constant $\dfrac{1}{3}$. For any $x\neq y$ in $X$, suppose  that $x\neq 0$ and therefore $x=g(n)$,  for some $n\in\mathbb{Z}$. Further,  $f^{n}(x)=1$ and $f^{n}(y)\leq 1/2$ so $d(f^{n}(x),f^{n}(y))> 1/3$.  Note that there are only finitely many values of $n$ for which $d(f^n(x),f^n(y))>1/3$ as this is possible only if at least one  of $f^n(x)$ and $f^n(y)$ is either $1$ or $1/2$. But  there is at most one value of $n$ for which $f^n(x)=1$ and at most one value of $n$ such that $f^n(x)=1/2$ and similarly for $y$.

\bigskip
\noindent In fact, it is know that if there exists a continuous, one-to-one, positively metric expansive map a compact metric space, then the space is always finite. There are couple of proofs known for this result. For instance, see \cite{an,ck}. In the following, we give  another proof of it.

\medskip
\begin{proposition}
Let $X$ be a compact metric space and $f:X \longrightarrow X$ be a continuous, one-to-one, positively metric expansive map. Then $X$ is a finite set.
\end{proposition}
\begin{proof}
Let $X$ be an infinite compact metric space and $f$ be a continuous, one-to-one, positively metric expansive map with expansive constant $c>0$. One can check that there is $m\in\mathbb{N}$ such that for every $k\geq0$ if $d(f^i(x), f^i(y))\leq c$ for $i=k+1, k+2, \ldots, k+n$, then $d(f^i(x), f^i(y))\leq c$ for all $0\leq i\leq k+n$. Since $X$ is an infinite $T_1-$space,  it follows  by Proposition \ref{finite} for every finite open cover $\mathcal{U}$, there is an open cover $\mathcal{V}$ such that for every $n\in\mathbb{N}$, $f^{-n}(\mathcal{U})\nprec \mathcal{V}$. Take finite open cover $\mathcal{U}$ containing open sets of the form
\begin{equation*}
U(z)=\left\{x\in X:d(f^i(x), f^i(y))\leq \frac{c}{2} \text{ \ for \ } i=1, 2, \ldots, m \right\}
\end{equation*}
Then there is open cover $\mathcal{V}$ such that for every $n\in\mathbb{N}$, $f^{-n}(\mathcal{U})\nprec \mathcal{V}$. Let $\delta>0$ be a Lebesgue number of open cover $\mathcal{V}$. Then for every $k\in\mathbb{N}$ there exist $x_k, y_k\in f^{-k}(\mathcal{U})$ such that $d(x_k, y_k)>\delta$. Consider sequences $\{x_k\}$ and $\{y_k\}$ is a compact metric space $X$. Suppose $\{x_k\}$ converges to $x$ and $\{y_k\}$ converges to $y$. Since $d(x_k, y_k)>\delta$ it follows that $x\neq y$. Also $x_k, y_k\in f^{-k}(\mathcal{U})$ implies that $f^k(x_k)$ and $f^k(y_k)$ lies in the same set $U(z)$. Thus $d(f^i(x_k), f^i(y_k))\leq c$ for $i=k+1, k+2, \ldots, k+m$. Since $d(f^i(x_k), f^i(y_k))\leq c$ for $i=0, 1, \ldots, k+m$, we have  $d(f^i(x), f^i(y))\leq c$ for all $i \geq 0$ while $x\neq y$. This contradicts positive metric expansivity of $f$.
\end{proof}

\medskip


\section{Topologically expansive homeomorphism}
\bigskip
\noindent In this section we study expansivity on uniform spaces. The notion was first defined in \cite{pram}. Let $X$ be an uniform space with uniformity $\U$ and $f:X \longrightarrow X$ be a homeomorphism. For an entourage $D \in \U$ let
$$\Gamma_D(x, f)=\left\{y  \: : (f^{n}(x), f^{n}(y))\in D, \: \forall n\in\mathbb{Z}\right\}.$$

\medskip
\begin{definition}\label{deft}
Let $X$ be an uniform space with uniformity $\U$. A homeomorphism $f:X \longrightarrow X$ is said to be a \emph{topologically   expansive homeomorphism}, if there exists an entourage $A\in\mathcal{U}$, such that for every $x\in X$,
 $$\Gamma_A(x, f)=\{x\}$$
Entourage $A$ is called  \emph{an expansive entourage}.
\end{definition}

\medskip
\noindent Note that if $B \in \U$ is a closed entourage such that $B \subset A$, then $B$ is also an expansive entourage of $f$.  Since every  entourage $A\in \mathcal{U}$ is also a neighborhood of $\Delta_X$, it follows that every topologically expansive  homeomorphism is an expansive homeomorphism. But in general converse need not be true can be observed from the following Example:

\medskip
\begin{example}
Consider $\R$ with the uniformity $\U$ generated by usual metric on $\R$. Then the translation  $T$ defined on $\R$ by $T(x)= x+1$ is an expansive homeomorphism with an expansive neighbourhood $N= \{(x, y)\in\mathbb{R}^2: |x-y|\leq e^{-x}\}$. Note that $N \notin \U$. In fact, it is easy to  observe that $T$ is not topologically expansive.
\end{example}

\smallskip
\begin{example}\label{ex1}
Consider $\R$ with uniformities $\U_\rho$ and $\U_d$ as given in Example \ref{compa}. Define a homeomorphism  $f:\R \longrightarrow \R$ by $f(x)= x+ln2$. Then it can be easily verified that $f$ is topologically expansive for a closed entourage $A \in \U_\rho$ but not for any closed  entourage $D \in \U_d$. Further, observe that $f$ is metric expansivity with respect to metric $\rho$ but is not metric expansive with respect to metric $d$.
\end{example}

\medskip
\noindent From Example \ref{ex1} it can be concluded that the notion of topological expansivity depends on the choice of uniformity on the space and the notion of metric expansivity depends on the metric of the space. In the following Remark we observe certain results related to topological expansivity as a consequence of expansivity.


\medskip
\begin{remark}
Let $X$ be a uniform space with uniformity $\U$ and let $f : X \longrightarrow X$ be a homeomorphism.
\begin{enumerate}
  \item Suppose $X$ is a locally compact, paracompact uniform space. Since every topologically expansive homeomorphism is an expansive homeomorphism, it follows from Lemma 9 of \cite{T.D} that there is a proper expansive neighborhood for $f$. Note that this neighborhood need not be an entourage. Recall, a set $M\subseteq X\times X$ is proper if for every compact subset $A$ of $X$, the set $M[A]$ is compact.
  \item Let $f$ be topologically expansive homeomorphism. Then by Proposition 13 of \cite{T.D} it follows that for each $n \in \N$,  $f^n$ is expansive. Note that this $f^n$ need not be general topologically expansive. For instance, let $\U$ be the usual uniformity on $[0, \infty)$ and $f:[0, \infty)\longrightarrow [0, \infty)$ be as homeomorphism constructed by Bryant and Coleman in \cite{b.f}. Then it is easy to verify that $f$ is topologically expansive but $f^n$ is not topologically expansive, for any $n>1$.
  \item Let $X$ be a uniform space with uniformity $\U$ and $Y$ be a uniform space with uniformity $\mathcal{V}$. Suppose $f : X \longrightarrow X$ is topologically expansive and $h : X \longrightarrow Y$ is a homeomorphism. Then by Proposition 13 of \cite{T.D} if follows that $h\circ f\circ h^{-1}$ is expansive on $Y$. However, the homeomorphism $h\circ f\circ h^{-1}$ need not be topologically expansive. For instance,  let $\U_\rho$ and $\U_d$ be uniformities on $\R$ as defined in Example \ref{compa}. Consider the identity homeomorphism $h : \R \longrightarrow \R$, where the domain $\R$ is considered with uniformity $\U_\rho$ whereas codomain is considered with the uniformity $\U_d$.  Then as observed in Example \ref{ex1} $f(x)=x + \ln (2)$ is topologically expansive with respect to $\U_\rho$ but $h\circ f\circ h^{-1}$ is not topologically expansive with respect to $\U_d$.
    \end{enumerate}
\end{remark}

\medskip
\noindent Observe here that in each of the above Example, $f$ is not uniformly continuous. In the following we show that Remarks above are true if the maps are uniformly continuous. Recall, a map $f :X \longrightarrow X$ is uniformly continuous relative to the uniformity $\U$ if for every entourage $V\in \mathcal{U}$, $(f\times f)^{-1}(V)\in \mathcal{U}$.

\medskip
\begin{proposition}\label{pp}
\begin{enumerate}
  \item Let $X$ be a uniform space with  uniformity $\U$. Suppose both $f$ and $f^{-1}$ are uniformly continuous relative to $\U$. Then $f$ is topologically expansive if and only if $f^n$ is topologically expansive, for all $n\in \Z\backslash \{0\}$.

  \item Let $X$ be a uniform space with uniformity $\U$ and $Y$ be a uniform space with uniformity $\mathcal{V}$. Suppose $h:X \longrightarrow Y$ is a homeomorphism such that both $h$ and $h^{-1}$ are uniformly continuous.  Then $f$ is topologically expansive on $X$ if and only if  $h\circ f\circ h^{-1}$ is topologically expansive on $Y$.
\end{enumerate}
\end{proposition}

\medskip
\noindent Since the proof of the Proposition \ref{pp} is similar to the proof of Proposition 13 in \cite{T.D}, we omit the proof.  In spite of expansivity,  in the following Proposition we show that if a uniform space admits a topologically expansive homeomorphism, the space is always Hausdorff space.

%
%

\medskip
\begin{proposition}\label{t1}
Let $X$ be a uniform space with uniformity $\U$ and let $f : X \longrightarrow X$ be a topologically expansive homeomorphism. Then $X$ is always a Hausdorff space.
\end{proposition}
\begin{proof}
Let $D$ be an expansive entourage of $f$.  Since $\U$ is a uniformity on $X$ there exists a symmetric set $E\in\U$,  such that  $$EoE\subseteq D.$$
Given two distinct points $x$ and $y$ of $X$, by topological expansivity of $f$ there exists $n$ in $\mathbb{Z}$, such that $(f^{n}(x), f^{n}(y))\notin D$. But this implies $$(f^{n}(x), f^{n}(y))\notin E\circ E.$$
Let $U=f^{-n}\left(E[f^{n}(x)]\right)$ and $V=f^{-n}(E[f^{n}(y)])$. Then $int(U)$ and $int(V)$ are open subsets of $X$ with $x \in int(U)$ and $y \in int(V)$. Further, $U\cap V=\emptyset$. For, if $t \in U\cap V$, then $f^n(t) \in E[f^{n}(x)] \cap E[f^{n}(y)]$. But this implies that   $(f^{n}(x), f^{n}(y))\in E\circ E$, which is a contradiction.  Hence $X$ is a Hausdorff space.
\end{proof}

\medskip
\noindent Following Corollary is a consequence of just Proposition \ref{t1} and Lemma \ref{has}.

\medskip
\begin{corollary}
If uniform space $X$ admits a  topological expansive homeomorphism then $X$ is a regular space.
\end{corollary}

\medskip
\noindent Recall, for a compact Hausdorff space, $X$, all uniformities generates a same topology on the space and therefore it is sufficient to work with the natural uniformity on $X$. Hence as consequence of Proposition \ref{t1} we can conclude the following:

\medskip
\begin{corollary}
Topological expansivity on a compact Hausdorff uniform space does not depend on choice of uniformity on the  space.
\end{corollary}

\medskip
\noindent Since every compact metric space admits a unique uniform structure, it follows that on compact metric space metric expansivity, topological expansivity and expansivity are equivalent.

\medskip
\noindent Let $X$ be a uniform space with uniformity $\U$. A cover $\mathcal{A}$ of a space $X$ is a uniform cover if there is $U\in\mathcal{U}$ such that $U[x]$ is a subset of some member of the cover for every $x\in X$, equivalently,  $\{U[x] : x\in X\}\prec \mathcal{A}$. It is known that every open cover of a compact uniform space is uniform cover. For instance,  see Theorem 33 in \cite{kelly}.

\medskip
\noindent  Let $X$ be a topological space and $f : X \longrightarrow X$ be an orbit expansive homeomorphism with an orbit expansive covering $\mathcal{A}$. Equivalently, $f$ is orbit expansive if for every subset $B$ of $X$, $f^{n}(B)\prec \mathcal{A}$ for all $n\in\mathbb{Z}$, then $B$ is  singleton. In the following we show that on compact uniform space, topological expansivity is equivalent to orbit expansivity:

\medskip
\begin{proposition}\label{t2}
     Let $X$ be a compact uniform space with uniformity $\U$. Then $f : X \longrightarrow X$ is a topologically expansive homeomorphism if and only if it is an orbit expansive homeomorphism.
\end{proposition}
   \begin{proof}
   Let $f$ be a topologically expansive homeomorphism with an expansive entourage  $D$, $D \in \U$. Choose $E\in \mathcal{U}$ such that $EoE\subseteq D$. Now, $E \in \U$ and $\U$ is a uniformity. Therefore $E$ contains diagonal and hence the collection $\left\{E[x] : x\in X\right\}$ is a cover of $X$ by neighbourhoods. But $X$ is compact. Let $\mathcal{A}$ be a finite subcover of $\left\{E[x] : x\in X\right\}$. We show that $\mathcal{A}$ is an orbit expansive covering for $f$.  For   $x, y \in X$ suppose that for each $n \in \Z$,  $\left\{f^{n}(x), f^{n}(y)\right\}\prec \mathcal{A}$. But this implies that for each $n \in \Z$,
   $$(f^n(x), f^n(y))\in EoE\subseteq D.$$
   Since $D$ is expansive entourage, it follows that $x=y$. Hence $\mathcal{A}$ is an orbit expansive covering.

\medskip
\noindent Conversely, let $\mathcal{A}$ be an orbit expansive covering of $f$. Since $X$ is a compact uniform space, $\mathcal{A}$ is a uniform cover. Therefore there exists $U\in\mathcal{U}$ such that $\{U[x]: x\in X\}\prec \mathcal{A}$. Since family of closed members of a uniformity $\mathcal{U}$ is a basis of $\mathcal{U}$, there is a closed member $D\in\mathcal{U}$ such that $D\subseteq U$. We claim that $D$ is an expansive entourage of $f$.   For $x,y\in X$ and for all $n \in \Z$, suppose
$$(f^n(x), f^n(y))\in D.$$
Therefore, for each $n \in \Z$,
$$\{f^n(x), f^n(y)\}\subseteq U[f^{n}(x)]. $$
This further implies that
$$\{f^n(x), f^n(y)\}\prec \{U[x]: x\in X\}\prec \mathcal{A}.$$
But $\mathcal{A}$ be an orbit expansive covering of $f$  and therefore $x=y$. Hence $f$ is topologically expansive with expansive entourage $D$.
\end{proof}

\medskip
\noindent In \cite[Theorem 2.7]{art} authors showed that if a compact Hausdorff topological space admits an orbit expansive homeomorphism then it is metrizable. Therefore by Proposition \ref{t1} and Proposition \ref{t2}, we have:

\medskip
\begin{corollary}\label{t3}
If a compact uniform space admits a topologically expansive homeomorphism, then it is always metrizable.
\end{corollary}

\medskip
\noindent Again as a consequence of Corollary \ref{t3}, it follows that topological expansivity is equivalent with metric expansivity and it does not depend uniformity. However the following example shows that Corollary \ref{t3} is false for non-compact Hausdorff uniform spaces.

\begin{example}\label{t45}
Consider $\mathbb{R}$ with the topology $\tau_\mathbb{R}$ whose base consists of all intervals $[x, r)$, where $x$ is a real number, $r$ is a rational number and $x<r$.  Then $\R$ with topology $\tau_\R$ is a non--compact, paracompact, Hausdorff and not  metrizable space.  Also, it is known that every paracompact Hausdorff space,  admits the uniform structure $\mathcal{U}$, consisting of all neighborhood of the diagonal. For instance, see \cite[Page 208]{kelly}. Hence if
$$    D= \{(x, y)\in \mathbb{R}\times \mathbb{R} : |x-y|<1\},$$
then $D\in \mathcal{U}$. Define $f:\R \longrightarrow \R $ by $f(x)= 3x$. Then it is easy to see that $f$ is topologically expansive with expansive entourage $D$.  Note that $\R$ with uniformity $\U$ is a non-compact Hausdorff space.
\end{example}


\section{Positively topologically expansive}

\medskip
\noindent In this Section we study a variant of a positively expansive map in a more general setting than uniform spaces. We note here that these results are always true if the space is a uniform space.

\medskip
\begin{definition}
Let $X$ be a topological space and let $f : X \longrightarrow X$ be a continuous function. Then $f$ is said to be a \emph{positively topological expansive}, if there is a neighborhood $U$ of the diagonal such that for any $x, y \in X$ with $x\neq y$ there exists a non--negative integer $n$ satisfying $(f^n(x), f^n(y))\notin U$. Neighborhood $U$ of the diagonal, $U\subseteq X\times X$ is called \emph{expansive neighborhood of $f$}.
 \end{definition}

\medskip
\noindent If $X$ is metric space and $f:X \longrightarrow X$ is positively metric expansive map with expansive constant $\delta>0$, then $U=\{(x, y): d(x, y)<\delta\}$ is an expansive neighborhood for $f$.  But the converse need not be true in general can be justified by the Example \ref{pe1}.

\medskip
\begin{example}\label{pe1}
Consider $X = [0,2]$ with usual metric. For $s \in [\sqrt{2}, 2]$, define a map $f_s : X \longrightarrow X$ as follows:
$$ f(x) = \left\{
\begin{array}{cc}
  sx, & 0 \leq x  \leq 1\\
  s(2-x), & 1 \leq x \leq 2
\end{array} \right. $$
Let $U_1$ be the subset of $[0,1]\times [0, 1]$ containing the diagonal of $[0, 1]$ and bounded by curves $h_1(x)=\frac{s+1}{2s}x^2$ and $h_2(x)=\sqrt{\frac{2s}{s+1}x}$. Further, Let $U_2$ be the subset of $[1,2]\times [1, 2]$ containing the diagonal of $[1, 2]$ and bounded by curves $g_1(x)=\frac{4s}{4s+2}+\frac{s+1}{4s+2}x^2$ and $g_2(x)=\sqrt{\frac{4s+2}{s+1}x-\frac{4s}{s+1}}$. If $U=U_1\cup U_2$, then it can be easily verified that each $f_s$ is positively topologically expansive. But it is known that there does not exists any metric positively expansive map on $X$ \cite{an}.
\end{example}

\medskip
\noindent In the following we give example of a map which is positively topological expansive for a neighborhood $U$ of diagonal which does not contain any subset $A_{\delta}=\{(x,y)\in X \times X : d(x,y) \leq \delta\}$.

\bigskip
\begin{example}
Consider $\R$ with usual topology. For each $x \in \R \backslash \Q$, let $\{x_i\}$ be a sequence in $\Q$ converging to $x$ in usual topology. Consider a new topology $\tau$ on $\R$ as follows:  If $y\in \R$ is a rational number then $\{y\}$ is open and if $y$ is an irrational number then basic open set about $y$ is given by $U_n(x)=\{y_i\}_{i=n}^\infty \cup \{y\}$, $n \in \N$. Note that intersection of 	any $U_n(x)$ with $\R \backslash \Q$ is singleton. Let $\delta > 0$ be any real number and $A_\delta =\left\lbrace (x, y) \in \R^2 : \left|x-y \right| \leq \delta \right\rbrace$. Set
$$U= \bigtriangleup_\R \cup \left(((\R \times \R) \backslash (\Q \times \Q))  \cap A_\delta \right) \cup \left((\R \times \R \backslash \Q)  \cap A_\delta \right).$$
Then $U$ is a neighborhood of $\bigtriangleup_\R$ and $U \subset A_\delta$. Define a map $f : \R \longrightarrow \R$ by $f(x) = 2x$. Then $f$ is positively expansive with expansive neighborhood $A_\delta$, for any $\delta >0$ and therefore $f$ is positively topological expansive with expansive neighborhood  $U$. Observe that for no $\delta$, $A_\delta \subset U$ as points of $A_\delta$ whose both coordinates are rational numbers are not in $U$.
\end{example}

\medskip
\noindent A neighborhood $U \subset X\times X$ of the diagonal $\Delta_X$ is said to be wide if there is a compact set $\dot{S} \subseteq X$ such that $U\bigcup \dot{S}\times X= X\times X$. In
other words, $U$ is wide if for any $x$ not in the compact set $\dot{S}$, the cross section $U[x]$ is in $X$. Also a set $M\subseteq X\times X$ is proper if for any compact subset $S$, the set $M[S]=\bigcap_{x\in S} M[x]$ is compact. For neighborhood $U$ of $\Delta_X$, denote
$$V_i^j(U)=\{(x, y)\in X\times X: (f^n(x), f^n(y))\in U \text{ \ for all \ } i\leq n\leq j\}. $$

\medskip
\begin{proposition}\label{a}
Let $f:X\to X$ be positively topological expansive.
\begin{enumerate}
  \item Let $X$ be a locally compact and paracompact space. Then $f$ has proper expansive neighborhood,
  \item Let $X$ be a locally compact and paracompact space and let $A$ be a proper expansive neighborhood for $f$. Then for every wide neighborhood $U$ of $\Delta_X$, there is $N\in\mathbb{N}$ such that $V_0^N(A)\subseteq U$.
  \item Let $X$ be a compact space and $U$ be expansive neighborhood for $f$. Then there is $m\geq 1$ such that $V^m_1(U)\subseteq V^m_0U)$.
      \item Let $X$ be a compact space and $U$ be expansive neighborhood for $f$. Then there is $m\geq 1$ such that for every $k\geq 0$ if $(f^k(x), f^k(y))\in V^m_1(U)$, then $(x, y)\in V^{m+k}_0(U)$.
\end{enumerate}
\end{proposition}
\begin{proof}
Proof of items (1) and (2) are similar to proof of Lemma 9 and Lemma 18 in \cite{T.D}. So we omit it here. We prove (3).

\medskip
\noindent Suppose for every $m\in\mathbb{N}$, there is $(x_m, y_m)\in V^m_1(U)$ such that $(x_m, y_m)\notin U$. Choose convergent sub sequences $x_{n_k}$ and $y_{n_k}$ such that $x_{n_k}$ converges to $x$ and $y_{n_k}$ converges to $y$. Then $x\neq y$ and for every every $i\geq 1$, $\left\lbrace (f^i(x_{n_k}), f^i(y_{n_k}))\right\rbrace $ will converge to $ (f^i(x), f^i(y))$. Now $x\neq y$ and  $(x_{n_k}, y_{n_k})\in V^{n_k}_1(U)$ imply $f(x)\neq f(y)$ and  $(f^i(f(x)), f^i(f(y)))\in U$ for all $i\geq 0$. This contradicts  positive topological  expansiveness of $f$.

\medskip
\noindent For Part (4). Fix $k\geq 0$, and apply item (3)with $f^j(x)$ and $f^j(y)$ in place of $x$ and $y$, for $j=k, k-1, \ldots, 1, 0$.
\end{proof}

\medskip
\noindent Let $X$ be a compact space and $U$ be expansive neighborhood for $f$. Choose neighborhood $D$ of $\Delta_X$ such that  $D\circ D\subseteq U$.  For every $z\in X$, take $U(z)=\{x\in X: (x, z)\in V^m_1(D)\}$. It is clear that $\{U(z): z\in X\}$ is an open cover for $X$. Since $X$ is compact, it has finitely many, say $N$, open sets of the form $U(z)$. Consider a finite subset containing $N+1$ points. If for every $k\geq 0$, there exist $x_k\neq y_k$ in this subset such that $(f^k(x_k), f^k(y_k))\in V^N_1(U)$. Hence by item (4) of Proposition \ref{a}, $(x_k, y_k)\in V^{m+k}_0(U)$. Since there are only finitely many pairs of these $N+1$ points, there exist $x\neq y$ such that for infinitely  many $k\in\mathbb{N}$ we have $(x, y)\in V^i_0(U)$ for $i=0, 1, 2, \ldots, n+k$ . This implies that $x=y$ that is a contradiction. Hence we have following:

\medskip

\begin{proposition}
Let $X$ be a compact space uniform space and $f : X \longrightarrow X$ be an injective continuous map. Suppose $f$ is positively topological expansive. Then $X$ is finite.
\end{proposition}
\bigskip
\noindent Let $X$ be a topological space and $f: X \longrightarrow X$ be a continuous map. A subset $A$ of $X$ is said to be \textit{invariant se}t if $f(A)=A$. For a subset $N$, let $Inv N$ denote the maximum invariant subset of $N$. Then
$$Inv N = \left\{x \in N : \exists \; ..., x_{-1}, x_0, x_1,... \in N \mbox{ such that } x = x_0 \mbox{ and } f(x_k) = x_{k+1}, \forall k \in \Z \right\}.$$

\medskip
\noindent In the following we show that the maximum invariant subset of a positively topological expansive map is the diagonal of space.

\medskip
\begin{proposition}\label{positive1}
	Let $X$ be a topological space and $f : X \longrightarrow X$ be a positively topologically expansive map with an expansive neighborhood $U$. Then the maximum invariant subset of $U$ is the diagonal of $X$.
\end{proposition}
\begin{proof}
Given that $U$ is an expansive neighborhood for map $f$. Therefore
$$\bigtriangleup_X \subset U \subset X\times X.$$
Let $(x,y) \in Inv \ U$, the largest invariant subset of $U$. Then there exists a sequence of points  $\left\lbrace(x_{i},y_{i}) : i \in \Z \right\rbrace$ in  $U$ such that
$$(x,y) = (x_0, y_0) \mbox{ and } \left(f(x_k), f(y_k)\right) = \left(x_{k+1},y_{k+1}\right),  \forall \ k \in \Z .$$
Therefore, for $n \geq 0$,
$$\left(x_{n+1}, y_{n+1}\right)=\left(f^n(x), f^n(y) \right)\in U.$$	
But $U$ is expansive neighborhood for $f$ will imply that $x=y$. Hence,
$$Inv U \subset \bigtriangleup_X.$$
Thus the maximum invariant set in $U$ is $\bigtriangleup_X$.
\end{proof}

\medskip
\noindent In the following Proposition we give sufficient conditions to imply positively topological expansive  map.
\begin{proposition}\label{t1}
Let $X$ be a  topological space and $f:X \longrightarrow X$ be a continuous map.  Then $f$  is a positively topological expansive map if one of the following condition holds.
\begin{enumerate}
  \item There is a finite open cover $\mathcal{U}=\{ U_1,\ldots, U_m \}$ of $X$ such
that if for all $n \in \N$, $\{f^n(x), f^n(y)\}\prec \mathcal{U}$, then $x =y$.
  \item If  $T_1-$ space $X$ has an open cover $\mathcal{U}=\left\{U_i : 0 \leq i \leq m\right\}$  such
that for every open cover $\mathcal{V}$ there is $n\in\mathbb{N}$ such that $f^{-n}(\mathcal{U})\prec \mathcal{V}$.
\item If  $T_1-$space $X$ has an open cover $\mathcal{U}=\left\{U_i : 0 \leq i \leq n\right\}$ such that for every open cover $\mathcal{V}$ of $X$, there is an $N\in\mathbb{N}$ satisfying $\bigwedge_{0\leq i\leq N}f^{-i}(\mathcal{U})\prec \mathcal{V}$.
\end{enumerate}
\end{proposition}
\begin{proof}
\begin{enumerate}
  \item suppose that there is a finite open cover $\mathcal{U}=\left\{ U_1,\ldots, U_m \right\}$ of $X$ satisfying the hypothesis. Take $U= \bigcup _{i=1}^m U_i\times U_i$. Then $U$ is a subset of $X \times X$ containing $\Delta_X$. For $x, y \in X$ and any $n \in \N$, suppose $(f^n(x), f^n(y))\in U$. Then for each $n \in \N$,
$$\{f^n(x), f^n(y)\}\prec \mathcal{U}.$$
But by hypothesis this implies that $x=y$. Hence $f$ is positively topological expansive.
  \item By item (1), it is sufficient to show that for finite open cover $\mathcal{U}=\{ U_0, U_1,\ldots, U_n \}$ of $X$ and for $x, y \in X$, with $x\neq y$, there is $n\in\mathbb{N}$ such that $\{f^n(x), f^n(y)\}\nprec \mathcal{U}$. Now, $X$ is $T_1-$space and therefore  $\mathcal{V}=\left\{X-\{x\}, Y-\{y\}\right\}$ is an open cover of $X$. Hence there exist  $n\in\mathbb{N}$  such that
 $$ f^{-n}(\mathcal{U})\prec \mathcal{V}.$$
 But $\{x, y\}\nprec \mathcal{V}$. Therefore hence
 $$\{x, y\}\nprec f^{-n}(\mathcal{U}).$$
 Hence the proof.
 \item Proof is similar to item (2)
\end{enumerate}
\end{proof}

\medskip
\noindent The next example shows that Proposition \ref{positive1} is not true if in items(2) and (3) we do not assume that the space is $T_1$.

\medskip
\begin{example}
Consider $X=[0,1)$ with the topology $\tau_X=\{(a, 1): a\geq 0\}\cup \{X, \emptyset\}$. Then $X$ is  compact but not $T_1$ with respect to topology $\tau_X$. Take open cover $\mathcal{U}= \{X\}$. Further, if $\mathcal{V}$ is an open cover of $X$, then $X\in \mathcal{V}$, hence  $\mathcal{U}\prec \mathcal{V}$. Therefore the identity map $id:X\longrightarrow X$ satisfies in Part (3) of Proposition \ref{t1}, but it is not positively topological expansive.
\end{example}

\medspace
\noindent A $T_1$ finite topological space $X$ is discrete and if $\mathcal{U}=\{\{x\}: x\in X\}$, then for every open cover $\mathcal{V}$, we have $\mathcal{U}\prec \mathcal{V}$. Hence Proposition \ref{positive1} implies that identity map $id:X\to X$ is positively topological expansive.

\medskip
\noindent Suppose $f$ is a positively topological expansive map with an expansive neighborhood $U$. Therefore, $n \in \N$, $(f^n(x), f^n(y))\in U$, then $x=y$. But this implies that  $\bigcap_{n\in\mathbb{N}} (f\times f)^{-n}(U)= \Delta_X$. Hence if expansive neighborhood for $f$ is closed, then $X$ is a Hausdorff space. In the following we give an example of a positively topological expansive map on a non--Hausdorff space and hence it may be happen that an expansive neighborhood is not a closed subset of $X\times X$.

\medskip
\begin{example}\label{pnh}
 Let $X$ be a compact metric space with topology $\tau$ generated by the metric and let $f:X\longrightarrow X$ be positive expansive map with fixed point $x_0\in X$. Then we can find an open cover  $\mathcal{U}=\left\{U_i : 0 \leq i \leq n\right\}$ such that for any other open cover $\mathcal{V}$,  $$\bigwedge_{i=0}^Nf^{-i}(\mathcal{U})\prec \mathcal{V}$$
 for some $N\geq 0$. Consider a new point $x_1$ and set $\overline{X}=X\cup \{x_1\}$, equipped with the topology
 $$\overline{\tau}= \tau \cup \left\{W\cup \{x_1\}: x_0\in W, W\in\tau\right\}\cup \left\{(W\setminus \{x_0\})\cup \{x_1\}: x_0\in W, W\in \tau\right\}.$$
 By \cite[Example 3.14]{art}, $\overline{X}$  is compact $T_1$- space. Moreover, if $x_0$ is not an isolated point of $X$, then $\overline{X}$ is not a Hausdorff space. Define a function $g:\overline{X}\longrightarrow \overline{X}$ by $g(x)=f(x)$ if $x\in X$ and $g(x_1)=x_1$. Then $g$ is a continuous map. Further, suppose that $x_0\in U_n$ and $x_0\notin U_k$ for all $k=1, 2, \ldots, n-1$. Define the open set
 $$U_{n+1}= U_n\setminus  \{x_0\}\cup \{x_1\}.$$
 Since $x_0, x_1$ are fixed points of $g$, it follows that $x_0, x_1\notin g^j(U_k)$ for all $j=0, 1, \ldots,$ and all $k=1, 2, \ldots, n-1$. We show that if $\mathcal{F}= \{U_1, U_2, \ldots, U_n, U_{n+1}\}$, then for every open cover $\overline{\mathcal{V}}$ of $\overline{X}$ there is $N\in\mathbb{N}$ satisfying $$\bigwedge_{i=0}^Ng^{-i}(\mathcal{F})\prec \overline{\mathcal{V}}.$$
 Let $\overline{\mathcal{V}}=\{V_1, V_2, \ldots, V_r\}$ be an arbitrary cover of $\overline{X}$ with $x_0\in V_{r-1}$ and $x_1\in V_r$. For every $i=1, \ldots, r$, take $W_i= V_i-\{x_0, x_1\}$ and define
 $$W_{r+1}= ((V_{r-1}\cap V_r)\cup \{x_0\})\setminus \{x_1\}.$$
 Then $W= \{W_1, W_2, \ldots, W_r, W_{r+1}\}$ is an open cover of $X$ and $\mathcal{W}\prec \overline{\mathcal{V}}$. But this implies there is an $N\in\mathbb{N}$ such that
 $$\bigwedge_{i=0}^Nf^{-i}(\mathcal{U})\prec \mathcal{W}.$$
 Take a sequence $\left\{U_{k(i)} : 0 \leq i \leq N\right\} \subseteq \mathcal{F}$. Suppose that
 $$x_1\notin \bigcap_{i=0}^Ng^{-i}(U_{k(i)}).$$
 Then
 $$\bigcap_{i=0}^Ng^{-i}(U_{k(i)})\subseteq \bigcap_{i=0}^Nf^{-i}(U_{k(i)})\prec \mathcal{W}\prec \overline{\mathcal{V}}.$$
 If $x_1\in \bigcap_{i=0}^Ng^{-i}(U_{k(i)})$, then
 $$U_{k(i)}= U_{n+1}=U_n\setminus  \{x_0\}\cup \{x_1\}.$$
 This implies that
    $$\bigcap_{i=0}^Ng^{-i}\left(U_{k(i)}\right)=\left(\bigcap_{i=0}^Nf^{-i}(U_n)\setminus \{x_0\}\right)\cup \{x_1\}\prec \mathcal{W}\cup \{V_r\}\prec \overline{V}.$$
Thus in any case by Proposition \ref{t1}, $g:\overline{X}\to \overline{X}$ is positively topological expansive, while $\overline{X}$ is a compact non- Hausdorff space.
 \end{example}

\end{large}

\begin{thebibliography}{99}


\bibitem{art} M. Achigar, A. Artigue, I. Monteverde, \emph{Expansive homeomorphisms on non--Hausdorff spaces}, Top.  \& its  Appl., Vol.  \textbf{207} (2016), 109--122.

\bibitem{an} N. Aoki, K. Hiraide, \emph{Topological Theory of Dynamical Systems}, North-Holland Math. Library, Vol. 52, North-Holland Publishing Co., Amsterdam, 1994.

%


\bibitem{ali} A. Barzanouni, \emph{Finite expansive homeomorphisms}, Top. \& its Appl., Vol.  \textbf{253} (2019), 95--112.

\bibitem{ase} A. Barzanouni,     M. S. Divandar,  E. Shah,  \textit{On Properties of Expansive Group Actions}, to appear in Acta Mathematica Vietnamica.

\bibitem{rb} R. Bowen, \textit{Entropy--expansive maps}, Trans. of the AMS, Vol. \textbf{164} (1972), 323--331.

\bibitem{b.f} B.F. Bryant, \emph{D.B. Coleman, Some expansive homeomorphisms of the reals}, Am. Math. Mon., Vol. \textbf{73} (1966) 370--373

\bibitem{ck} E. M. Coven,  M. Keane,  \textit{Every compact metric space that supports a positively expansive homeomorphism is finite}, IMS Lecture Notes–Monograph Series Dynamics \& Stochastics Vol. \textbf{48} (2006) 304–305

\bibitem{pram} P. Das,  T. Das. \emph{Various types of shadowing and specification on uniform spaces}, J. Dyn. Control Syst., Vol. \textbf{24} (2018), No. 2, 253--267.

\bibitem{T.D} T. Das, K. Lee, D. Richeson, J. Wiseman. \textit{Spectral decomposition for topologically Anosov homeomorphisms on noncompact and non-metrizable spaces}, Top. \& its Appl.,  Vol. \textbf{160}(2013), 149--158.

\bibitem{HK-2}
H. Kato,
\emph{The nonexistence of expansive homeomorphisms of Peano continua in the plane},
Top. \& its Appl., Vol. \textbf{34}~(1990), no. 2, 161--165.

\bibitem{Hi0}
K. Hiraide,
\emph{Expansive homeomorphisms of compact surfaces are pseudo-Anosov},
Osaka J. Math., \textbf{27}~(1990), no. 1, 117--162.

\bibitem{HK} H. Kato,  \textit{Continuum--wise expansive homeomorphisms}, Canad. J. Math.,  Vol. \textbf{45} (1993), 576--598.


\bibitem{kelly} John L. Kelley, \emph{General Topology}, D. Van Nostrand Company, Inc., Toronto, New York, London, 1955.

\bibitem{Kou} M. Kouno,
\emph{On expansive homeomorphisms on manifolds},
J. Math. Soc. Japan, Vol. \textbf{33}~(1981), No. 3, 533--538.

\bibitem{rm}
R. Ma\~{n}\'e,
\emph{Expansive homeomorphisms and topological dimension},
Trans. Amer. Math. Soc., Vol. \textbf{252}~(1979), 313--319.

%

\bibitem{wr} W. L. Reddy, \textit{Point--wise expansion homeomorphisms}, J. Lond. Math. Soc., Vol. \textbf{2} (1970),  232--236.

\bibitem{D.R} D. Richeson, J. Wiseman, \textit{Positively expansive dynamical systems}. Top. \& its Appl., Vol. \textbf{154} (2007), 604--613.

\bibitem{SS} S. Schwartzman, \emph{On Transformation Groups}, Dissertation, Yale University, 1952.

\bibitem{es} E. Shah, \emph{Positively expansive maps on} $G-$\emph{spaces}, J.  Indian Math. Soc., Vol. \textbf{72} (2005), 91--97.

\bibitem{U} W.R. Utz, \textit{Unstable homeomorphisms}, Proc. Amer. Math. Soc., Vol. \textbf{1} (1950),  769--774.

\bibitem{Wei} A. Weil, \emph{Sur les espaces $\grave{a}$ structure uniforme et sur la topolgie g$\grave{e}$n$\grave{e}$rale}, Hermann, Paris, 1938.


\end{thebibliography}
\end{document}